\documentclass[12pt]{article}
\usepackage{amsmath, amssymb,amsthm}
\usepackage{graphicx}

\usepackage{amsthm}

\usepackage{enumerate}

\newcommand{\ex}{\mathrm{ex}}
\newcommand{\Ex}{\mathrm{Ex}}
\newcommand{\co}{\mathrm{Co}}
\newcommand{\ro}{\mathrm{Ro}}

\newtheorem{theorem}{Theorem}
\newtheorem{lemma}{Lemma}

\newtheorem{fact}{Fact}

\def\le{\leqslant}
\def\ge{\geqslant}
\usepackage{lineno}

\begin{document}

\title{One more Tur\'an number and  Ramsey number for the loose 3-uniform path of length three}

\author{Joanna Polcyn
\\A. Mickiewicz University\\Pozna\'n, Poland\\{\tt joaska@amu.edu.pl}}

\date{\today}

\maketitle

\begin{abstract}
Let $P$ denote  a 3-uniform hypergraph consisting of 7 vertices $a,b,c,d,e,f,g$ and 3 edges
$\{a,b,c\}, \{c,d,e\},$ and $\{e,f,g\}$. It is known that the $r$-color Ramsey number for $P$ is
 $R(P;r)=r+6$  for $r\le 9$. The proof of this result relies on a careful analysis
of the Tur\'an numbers for $P$. In this paper, we refine this analysis  further and compute the fifth order Tur\'an number for $P$, for all $n$.
Using this number for $n=16$, we  confirm the formula
$R(P;10)=16$.
\end{abstract}

\section{Introduction}\label{intro}
For the sake of brevity, 3-uniform hypergraphs will be called here \emph{ $3$-graphs}.
 Given a family of $3$-graphs $\mathcal F$, we say that a $3$-graph $H$ is  \emph{$\mathcal F$-free}
if for all $F \in \mathcal F$ we have  $H \nsupseteq F$.

 For a family of $3$-graphs $\mathcal F$ and an integer $n\ge1$, the \textit{ Tur\'an number of the 1st order}, that is, the ordinary Tur\'an number, is defined as
    $$
    \ex(n;\mathcal F)=\mathrm{ex}^{(1)}(n; \mathcal F)=\max\{|E(H)|:|V(H)|=n\;\mbox{ and $H$ is
    $\mathcal F$-free}\}.
    $$
 Every $n$-vertex $\mathcal F$-free $3$-graph with $\ex^{(1)}(n;\mathcal F)$
 edges is called  \emph{1-extremal for~$\mathcal F$}.
  We denote by $\mathrm{Ex}^{(1)}(n;\mathcal F)$
  the family of all, pairwise non-isomorphic, $n$-vertex $3$-graphs which are 1-extremal for $\mathcal F$.
Further,  for an  integer $s\ge1$,
 \textit{the Tur\'an number of the $(s+1)$-st order} is defined as
\begin{eqnarray*}
\mathrm{ex}^{(s+1)}(n;\mathcal F)=\max\{|E(H)|:|V(H)|=n,\; \mbox{$H$ is
    $\mathcal F$-free, and }\\
 \forall H'\in \mathrm{Ex}^{(1)}(n;\mathcal F)\cup...\cup\mathrm{Ex}^{(s)}(n;\mathcal F),  H\nsubseteq H'\},
\end{eqnarray*}
if such a $3$-graph $H$ exists. Note that if $\mathrm{ex}^{(s+1)}(n;\mathcal F)$ exists then, by definition,
\begin{equation}\label{decrease}
\mathrm{ex}^{(s+1)}(n;\mathcal F)<\mathrm{ex}^{(s)}(n;\mathcal F).
\end{equation}

An $n$-vertex $\mathcal F$-free $3$-graph $H$ is called \textit{(s+1)-extremal for} $\mathcal F$ if $|E(H)| = \mathrm{ex}^{(s+1)}(n;\mathcal F)$ and  $\forall H'\in \mathrm{Ex}^{(1)}(n;\mathcal F)\cup...\cup\mathrm{Ex}^{(s)}(n;\mathcal F),  H\nsubseteq H'$; we denote by $\mathrm{Ex}^{(s+1)}(n;\mathcal F)$ the family of $n$-vertex $3$-graphs which are $(s+1)$-extremal for $\mathcal F$.
 In the case when $\mathcal F=\{F\}$, we will write $F$ instead of  $\{ F\}$.

\emph{A loose 3-uniform path of length 3} is a 3-graph $P$ consisting of 7 vertices, say,
$a,b,c,d,e,f,g$, and 3 edges $\{a,b,c\}, \{c,d,e\},$ and $\{e,f,g\}$.
The \textit{Ramsey number} $R(P;r)$ is the least integer $n$ such that  every $r$-coloring
of the edges of the complete $3$-graph $K_n$ results in a monochromatic copy of $P$. Gyarfas and Raeisi \cite{GR} proved, among many other results, that $R(P;2)=8$. (This result was later extended to loose paths of arbitrary lengths, but still $r=2$, in \cite{Omidi}.)
Then Jackowska \cite{J} showed that $R(P;3)=9$  and $r+6\le R(P;r)$ for all $r\ge3$.
In turn, in \cite{JPR}, \cite{JPRr}, and \cite{PR}, the Tur\'an numbers of the first four orders,  $\ex^{(i)}(n;P)$, $i=1,2,3,4$, have been determined for all feasible values of $n$. Using these numbers, in \cite{JPRr} and \cite{PR}, we were able to compute the Ramsey numbers $R(P;r)$ for $4\le r\le9$.

 \begin{theorem}[\cite{GR, J,JPRr,PR}]\label{past}
    For all $r\le 9$, $R(P; r)=r+6$.
\end{theorem}

In this paper we determine, for all $n\ge 7$, the  Tur\'an numbers for $P$ of the fifth order, $\ex^{(5)}(n;P)$. This allows us to compute one more Ramsey number.

 \begin{theorem}\label{main}
     $R(P; 10)=16$.
\end{theorem}

It seems that in order to make a further progress in computing the Ramsey numbers $R(P;r)$, $r\ge11$, one would need to determine still higher order Tur\'an numbers $\mathrm{ex}^{(s)}(n;P)$, at least for some small values of $n$.

Throughout, we denote by $S_n$ the 3-graph on $n$ vertices and with $\binom {n-1}2$ edges, in which one vertex, referred to as \emph{the center}, forms  edges with all  pairs of the remaining vertices.
Every sub-3-graph of $S_n$ without isolated vertices is called \emph{a star}, while $S_n$ itself is called  \emph{the full star}. We  denote by $C$ \emph{the triangle}, that is, a 3-graph with six vertices $a,b,c,d,e,f$ and three edges $\{a,b,c\}$, $\{c,d,e\}$, and $\{e,f,a\}$. Finally, $M$ stands for a pair of disjoint edges. For a given 3-graph $H$ and a vertex $v\in V(G)$ we denote by $\deg_H(v)$ the number of edges in $H$ containing $v$.

 In the next section we state some known and new results on  Tur\'an numbers for $P$, including Theorem \ref{ex5} which provides a complete formula for  $\ex^{(5)}(n;P)$.  We also define conditional Tur\'an numbers and quote from  \cite{JPRr} and \cite{P} some useful
  lemmas  about the  conditional Tur\'an numbers  with respect to  $P$,  $C$,  $M$.
 Then, in Section \ref{proofRam},   we prove Theorem \ref{main}, while
 the remaining sections  are  devoted to proving Theorem \ref{ex5}.

\section{Tur\'an numbers}\label{turan}

We restrict ourselves exclusively to the case $k=3$ only.
A  celebrated result of Erd\H os, Ko, and Rado \cite{EKR} asserts, in the case of $k=3$, that for $n\ge 6$, $\mathrm{ex}^{(1)}(n;M)=\binom{n-1}2$. Moreover, for $n\ge7$,  $\mathrm{Ex}^{(1)}(n;M)=\{S_n\}$.
We will need the higher order versions of this Tur\'an number, together with its extremal families. The second of these numbers has been found by Hilton and Milner,  \cite{HM} (see \cite{FF2} and \cite{P} for a simple proof). For a given set of vertices $V$, with $|V|=n\ge 7$, let us define two special 3-graphs. Let $x,y,z,v\in V$ be four different vertices of $V$. We set
$$
G_1(n)=\left\{\{x,y,z\}\right\}\cup\left\{h\in \binom{V}{3}: v\in h, h\cap \{x,y,z\}\neq\emptyset\right\},
$$
$$
G_2(n)=\left\{\{x,y,z\}\right\}\cup\left\{h\in \binom{V}{3}: |h\cap \{x,y,z\}|=2\right\}.
$$
Note that for  $i\in \{1,2\}$, $M\not\subset G_i(n)$ and $|G_i(n)|=3n-8$.

\begin{theorem}[\cite{HM}]\label{ntif}
	For  $n\ge 7$,  $\mathrm{ex}^{(2)}(n;M)=3n-8$ and $\Ex^{(2)}(n;M)=\{G_1(n),G_2(n)\}$.
\end{theorem}
\noindent Later, we will use the fact that   $C\subset  G_i(n)\not\supset P$, $i=1,2$.

 Recently, the third order Tur\'an number for $M$ has been established for general $k$ by Han and Kohayakawa in \cite{HK}. Let $G_3(n)$ be the 3-graph on $n$ vertices, with distinguished vertices $x,y_1,y_2,z_1,z_2$ whose edge set consists of all edges spanned by $x,y_1,y_2,z_1,z_2$ except for $\{y_1,y_2,z_i\}$, $i=1,2$, and all edges of the form $\{x,z_i,v\}$, $i=1,2$, where $v\not\in\{x,y_1,y_2,z_1,z_2\}$.

\begin{theorem}[\cite{HK}]\label{ntntif}
	For  $n\ge 7$,  $\mathrm{ex}^{(3)}(n;M)=2n-2$ and  $\Ex^{(3)}(n;M)=\{G_3(n)\}$.
\end{theorem}

For $k=3$ we were able to take the next step and determine the next Tur\'an number for $M$.

\begin{theorem}[\cite{P}]\label{ntntntif}
	For  $n\ge 7$,  $\mathrm{ex}^{(4)}(n;M)=n+4$.
\end{theorem}

\bigskip
The number $\binom{n-1}2$ serves as the Tur\'an number for two other 3-graphs, $C$ and $P$.
 The Tur\'an
number $\ex^{(1)}(n;C)$ has been determined in \cite{FF} for $n\ge 75$ and later for all $n$ in
\cite{CK}.
\begin{theorem}[\cite{CK}]\label{c3}
    For $n\ge 6$, $\ex^{(1)}(n;C)=\binom{n-1}{2}$. Moreover, for $n\ge8$, \newline $\mathrm{Ex}^{(1)}(n;C)=\{S_n\}$.
\end{theorem}

 In \cite{JPR},  we filled an omission of \cite{FJS} and \cite{Kostochka} and calculated $\ex^{(1)}(n;P)$ for all $n$. Given two $3$-graphs $F_1$
and $F_2$, by $F_1\cup F_2$  denote a vertex-disjoint union of $F_1$ and $F_2$. If $F_1=F_2=F$ we will sometimes write $2F$ instead of $F\cup F$.

\begin{theorem}[\cite{JPR}]\label{ex1}
 $$\ex^{(1)}(n;P)=\left\{ \begin{array}{ll}
 \binom n3 & \textrm{ and $\quad Ex^{(1)}(n;P)=\{K_n\}\qquad\quad$\;\; for $n\le6,$ }\\
20 & \textrm{ and $\quad Ex^{(1)}(n;P)=\{K_6\cup K_1\}\quad\,\,$ for $n=7,$ }\\
\binom{n-1}{2} & \textrm{ and $\quad Ex^{(1)}(n;P)=\{S_n\}\qquad\quad$\;\;\; for $n\ge 8$.}
\end{array} \right.
$$
\end{theorem}

\bigskip

In \cite{JPRr} we have  completely determined the second order Tur\'an number  $\ex^{(2)}(n;P)$, together with the corresponding
2-extremal 3-graphs. \emph{A comet} $\co(n)$ is an $n$-vertex 3-graph  consisting of the complete 3-graph
$K_4$ and the full star $S_{n-3}$, sharing exactly one vertex which is the center of the
star (see Fig. \ref{FigR1}). This vertex is called  \emph{the center} of the comet, while the set of
the remaining three vertices of the $K_4$ is called its \emph{ head}.

\bigskip
\begin{figure}[!ht]
\centering
\includegraphics [width=7cm]{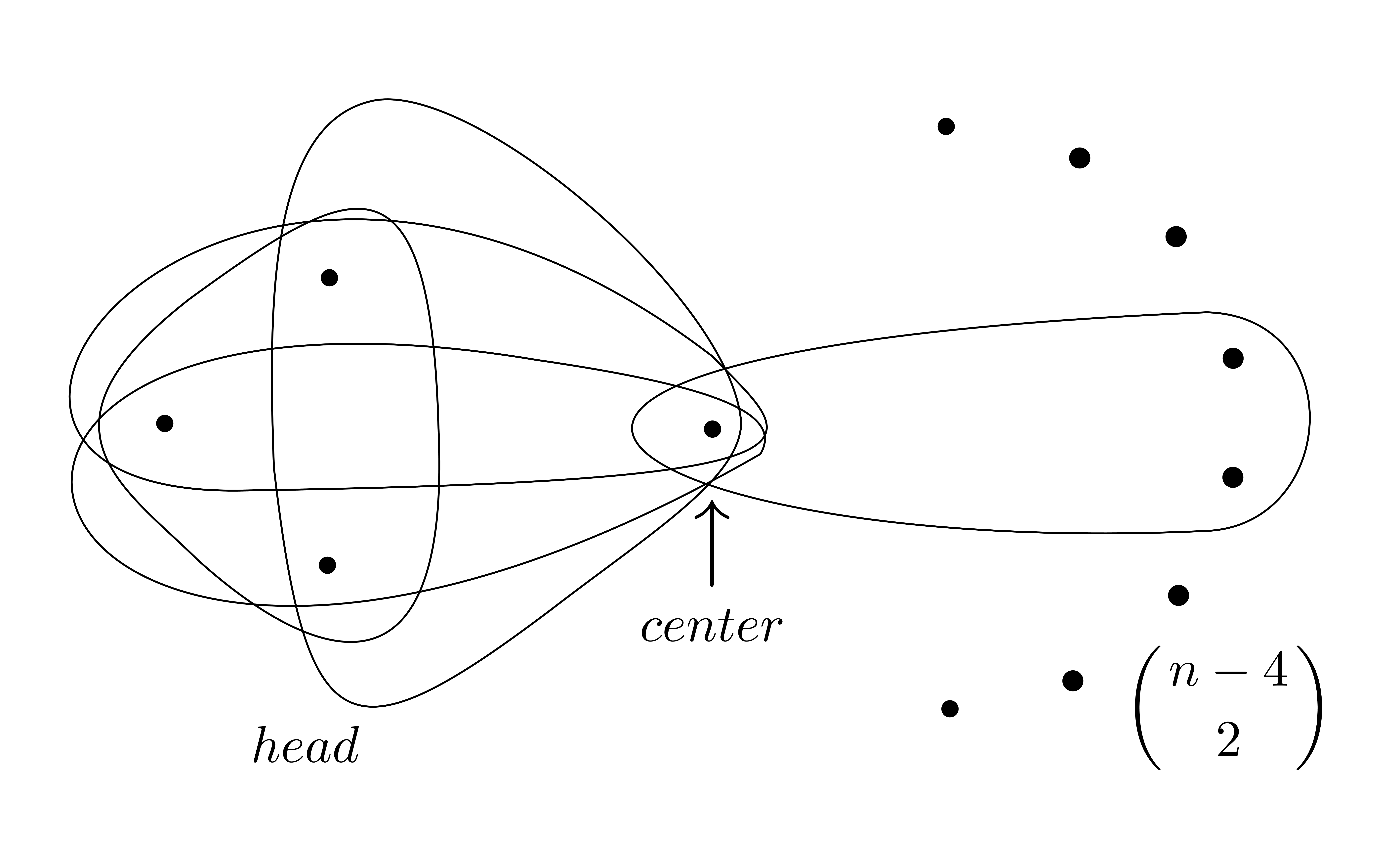}
\caption{The comet $\co(n)$} \label{FigR1}
\end{figure}

\begin{theorem}[\cite{JPRr}]\label{ex2}
    $$\ex^{(2)}(n;P)=\left\{ \begin{array}{ll}
    15 &\textrm{and} \quad  \Ex^{(2)}(n;P)=\{S_7\}\hskip 3.5cm \textrm{for $n=7$},\\
    20 + \binom {n-6}3 & \textrm{and}  \quad \Ex^{(2)}(n;P)=\{K_6 \cup K_{n-6}\}\hskip 1.05cm \textrm{for $8 \le n\le 12$},\\
    40 & \textrm{and}  \quad \Ex^{(2)}(n;P)=\{2K_6\cup K_1,\co(13)\} \hskip 0.6cm\textrm{for }n=13,\\
    4 + \binom{n-4}{2}&\textrm{and}\quad\Ex^{(2)}(n;P)=\{\co(n)\}\hskip 2.65cm\textrm{for }n\ge 14.
    \end{array} \right.
    $$
\end{theorem}

\noindent 

In \cite{JPRr} ($n=12$) and  in \cite{PR} (for all $n$), we   calculated the third order Tur\'an number for $P$.

\begin{theorem}[\cite{JPRr},\cite{PR}]\label{ex3}
	$$
	\mathrm{ex}^{(3)}(n;P)=
	\left\{\begin{array}{ll}
	3n-8 &\textrm{and} \quad \Ex^{(3)}(n;P)=\{G_1(n),G_2(n)\}\hskip 0.95cm\textrm{for }7\le n\le 10,\\
	25 &\textrm{and} \quad \Ex^{(3)}(n;P)=\{G_1(n),G_2(n),\co(n)\}\hskip 0.5cm\textrm{for } n= 11,\\
	32 &\textrm{and} \quad \Ex^{(3)}(n;P)=\{\co(n)\}\hskip 3cm\textrm{for }n=12,\\
	20 + \binom{n-7}{2}  &\textrm{and} \quad \Ex^{(3)}(n;P)=\{K_6\cup S_{n-6}\}\,\hskip 1.2cm\textrm{for }13\le n\le 14,\\
	4 + \binom{n-5}{2} &\textrm{and} \quad \Ex^{(3)}(n;P)=\{K_4\cup S_{n-4}\}\,\hskip 2.2cm\textrm{for }n\ge 15.
	\end{array}\right.
	$$
\end{theorem}

Surprisingly, as an immediate consequence we obtained also an exact formula for the 4th Tur\'an number for $P$.
Let $K_5^{+t}$ be the 3-graph obtained from $K_5$ by fixing two of its vertices, say $a,b$, and adding $t$ more vertices $v_1, v_2, \dots, v_t$ and $t$ edges $\{a,b,v_i\}$, $i=1,2,\dots, t$.

\begin{theorem}[\cite{PR}]\label{ex4}
	$$
	\mathrm{ex}^{(4)}(n;P)=
	\left\{\begin{array}{ll}
	12 &\textrm{and} \quad \Ex^{(4)}(n;P)=\{G_3(n),K_5^{+2}\}\hskip 2.6cm\textrm{for }n=7,\\
	2n-2 &\textrm{and} \quad \Ex^{(4)}(n;P)=\{G_3(n)\}\hskip 2,82cm\textrm{for }8\le n\le 9,\\
	20 &\textrm{and} \quad \Ex^{(4)}(n;P)=\{K_5\cup K_5\}\hskip 2,95cm\textrm{for }n=10,\\
	20 &\textrm{and} \quad \Ex^{(4)}(n;P)=\{G_3(n)\}\hskip 3.4cm\textrm{for }n=11,\\
	28 &\textrm{and} \quad \Ex^{(4)}(n;P)=\{G_1(n),G_2(n)\}\hskip 2.15 cm\textrm{for } n= 12,\\
	33 &\textrm{and} \quad \Ex^{(4)}(n;P)=\{K_6\cup G_1(n),K_6\cup G_2(n)\}\hskip 0.2cm\textrm{for }n=13,\\
	40 &\textrm{and} \quad \Ex^{(4)}(n;P)=\{2K_6\cup 2K_1,K_4\cup S_{10}\}\hskip 0.65cm\textrm{for }n=14,\\
	48  &\textrm{and} \quad \Ex^{(4)}(n;P)=\{\ro(n),K_6\cup S_{9}\}\,\hskip 1.74cm\textrm{for }n=15,\\
	3 + \binom{n-5}{2} &\textrm{and} \quad \Ex^{(4)}(n;P)=\{\ro(n)\}\,\hskip 3.35cm\textrm{for }n\ge 16.
	\end{array}\right.
	$$
\end{theorem}

The main Tur\'an-type result of this paper provides a complete formula for the fifth order Tur\'an number for $P$.

\begin{theorem}\label{ex5}
	
	\footnotesize
	$$
	\mathrm{ex}^{(5)}(n;P)=
	\left\{\begin{array}{ll}
	11 \quad \textrm{and} \quad \Ex^{(5)}(n;P)=\Ex^{(4)}(7;M)&\textrm{for }n=7,\\
	13 \quad\textrm{and} \quad \Ex^{(5)}(n;P)=\{K_5^{+3}\}&\textrm{for }n=8,\\
	14 \quad\textrm{and} \quad \Ex^{(5)}(n;P)=\{K_5^{+4},K_5\cup K_4\}\cup\Ex(9;\{P,C\}|M)&\textrm{for }n=9,\\
	19 \quad\textrm{and} \quad \Ex^{(5)}(n;P)=\{\co(10)\}&\textrm{for }n=10,\\
	19 \quad\textrm{and} \quad \Ex^{(5)}(n;P)=\{K_4\cup S_7\}&\textrm{for }n=11,\\
	25 \quad\textrm{and} \quad \Ex^{(5)}(n;P)=\{K_5\cup S_7, K_4\cup S_8\}&\textrm{for }n=12,\\
	32 \quad\textrm{and} \quad \Ex^{(5)}(n;P)=\{K_4\cup S_9, K_6\cup K_5^{+2}, K_6\cup G_3(7)\}&\textrm{for }n=13,\\
	39 \quad\textrm{and} \quad \Ex^{(5)}(n;P)=\{\ro(14)\}&\textrm{for } n= 14,\\
	46 \quad\textrm{and} \quad \Ex^{(5)}(n;P)=\{K_5 \cup S_{10}\}&\textrm{for }n=15,\\
	56 \quad\textrm{and} \quad \Ex^{(5)}(n;P)=\{K_6\cup S_{10}\}&\textrm{for }n=16,\\
	65 \quad \textrm{and} \quad \Ex^{(5)}(n;P)=\{K_5\cup S_{12},K_6\cup S_{11}\}&\textrm{for }n=17,\\
	10 + \binom{n-6}{2} \quad\textrm{and} \quad \Ex^{(5)}(n;P)=\{K_5\cup S_{n-5}\}&\textrm{for }n\ge 18.
	\end{array}\right.
	$$
	\normalsize
\end{theorem}

To determine Tur\'an numbers, it is sometimes useful to rely on Theorem \ref{ntif} and divide all 3-graphs into those which contain $M$ and those which do not.
To this end, it is convenient to define conditional Tur\'an numbers (see \cite{JPR, JPRr}).
     For a family of $3$-graphs $\mathcal F$, an $\mathcal F$-free $3$-graph $ G$, and an integer $n\ge |V(G)|$, the  \textit{conditional Tur\'an number} is defined as
     \begin{eqnarray*}
        \ex(n;\mathcal F|G)=\max\{|E(H)|:|V(H)|=n,\;
      \mbox{$H$ is
    $\mathcal F$-free, and } H\supseteq G \}
     \end{eqnarray*}
     Every $n$-vertex $\mathcal F$-free $3$-graph with $\ex(n;\mathcal F|G)$ edges and such
     that $H\supseteq G$  is called \emph{$G$-extremal for $\mathcal F$}.
     We denote by $\Ex(n;\mathcal F| G)$ the family of all  $n$-vertex $3$-graphs which are $ G$-extremal for $\mathcal F$.
 (If $\mathcal F=\{F\}$, we simply write $F$ instead of $\{F\}$.)

To illustrate the above mentioned technique, observe that for $n\ge7$
$$\mathrm{ex}^{(2)}(n;P)=\max\{ \ex(n; P|M), \mathrm{ex}^{(2)}(n;M)\}\overset{Thm\ref{ntif}
}{=}\max\{ \ex(n; P|M), 3n-8\}=\ex(n; P|M),$$
the last equality holding for  sufficiently large $n$ (see \cite{JPRr} for details).

In the  proof of Theorem \ref{ex5} we will use the following five lemmas, all proved in \cite{JPRr} and \cite{P}. For the first two we need one more piece of notation.
 If, in the above definition,  we restrict ourselves to connected $3$-graphs only (connected in the weakest, obvious sense) then the corresponding conditional Tur\'an number and the extremal family are denoted by
$ \ex_{conn}(n;\mathcal F|G)$ and  $\Ex_{conn}(n;\mathcal F| G)$, respectively.

\begin{lemma}[\cite{JPRr}]\label{spojny} For $n\ge7$,
    $$ \ex_{conn}(n; P|C)=3n-8\mbox{ and }\Ex_{conn}(n;P| C)=\{G_1(n),G_2(n)\}.$$
\end{lemma}

Lemma  \ref{spojny} as stated in \cite{JPRr} does not provide family $\Ex_{conn}(n;P| C)$. However, it is clear form its proof that the $C$-extremal 3-graphs are the same as in Theorem \ref{ntif}. We will need also another lemma, which is not stated explicitly in \cite{JPRr}, but it immediate results form the proof of the previous one. 

\begin{lemma}[\cite{JPRr}]\label{spojnyzCM} For $n\ge7$,
	$$ \ex_{conn}(n; P|\{C,M\})=n+5\mbox{ and }\Ex_{conn}(n;P|\{ C,M\})=\{K_5^{+(n-5)}\}.$$
	Moreover, if $H$ is $n$-vertex connected $P$-free 3-graph such that $C\subset H$ and $M\subset H$, then $H\subseteq K_5^{+(n-5)}$
\end{lemma}

\begin{lemma}[\cite{JPRr}]\label{PCM}
    $$
    \mathrm{ex}(n;\{P,C\}|M)=
    \left\{\begin{array}{ll}
    2n-4 &\qquad\qquad\qquad\qquad\qquad\qquad\qquad\;\;\;\;\textrm{for } 6\le n\le 9,\\
    20 &\qquad\qquad\qquad\qquad\qquad\qquad\qquad\qquad\;\;\textrm{for }n=10,\\
    4 + \binom{n-4}{2} &\textrm{and} \quad \Ex (n;\{P,C\}|M)=\{\co(n)\}\quad\quad\textrm{for } n\ge 11.
    \end{array}\right.
    $$
\end{lemma}

\begin{lemma}[\cite{JPRr}]\label{pcppm}
    For $n\ge 6$
    $$\ex(n;\{P,C,P_2\cup K_3\}|M)=2n-4,$$
where $P_2$ is a pair of edges sharing one vertex.
\end{lemma}

\begin{lemma}[\cite{P}]\label{2MC}
	For $n\ge 6$, $$\ex^{(2)}(n;\{M,C\})=\max\{10,n\}.$$
\end{lemma}

\section{Proof of Theorem \ref{main}}\label{proofRam}

As mentioned in the Introduction, Jackowska has shown in \cite{J}, that $R(P;r)\ge r+6$ for all $r\ge1$. We are going to show that $R(P;10)\le 16$.

	We will show that every 10-coloring of $K_{16}$ yields a monochromatic copy of $P$. The idea of the proof is to
gradually reduce the number of vertices and colors  (by one in each step),  until we reach a coloring which yields a monochromatic copy of $P$.

	
	Let us consider an arbitrary 10-coloring of $K_{16}$, $K_{16}=\bigcup_{i=1}^{10}G_i$, and assume that for each $i\in [10]$, $P\nsubseteq G_i$. Since $|K_{16}|=560$, the average number of edges per color is 56, and therefore, by Theorems \ref{ex1}--\ref{ex5}, either for each $i\in [10]$, $G_i=K_6\cup S_{10}$, or there exists a color, say $G_{10}$, contained in one of the $3$-graphs: $S_{16},\co(16),K_4 \cup S_{12},\ro(16)$. We will show, that the later case must occur. Indeed, for each vertex $v\in V(K_{16})$ we have $\deg_{K_{16}}(v)={15 \choose 2}=105$ whereas for $v\in V(K_6\cup S_{10})$, $\deg_{K_6\cup S_{10}}(v)\in \{10,36,8\}$ depending on weather $v$ is a vertex of $K_6$, the center of the star $S_{10}$ or an other vertex. Since we are not able to obtain an odd number as a sum of even numbers, we can not decompose $K_{16}$ into edge-disjoint copies of $K_6\cup S_{10}$. Let us turn back to $G_{10}$. No matter in which of the four 3-graph $G_{10}$ is contained, we remove the center of the star (or comet, or rocket) together with up to four more edges of $G_{10}$, so that we get rid of color 10 completely (note that some other colors can also be affected by this deletion).
	
	As a result, we obtain a 3-graph $H_{15}$ on 15 vertices, colored with 9 colors, $H_{15}=\bigcup_{i=1}^{9}G_i$, 
 with $|H(15)|\ge 451$ (with some abuse of notation we  will keep denoting the  subgraphs of $G_i$ obtained in each step again by $G_i$). The average number of edges per color is at least 50.1, and therefore there exists a color, say $G_9$, with $|G_9|\ge 51$. This time we use Theorems \ref{ex1}--\ref{ex3} to conclude that either $G_9\subset S_{15}$ or $G_9 \subset \co(15)$. In either case we remove the center and, in case of the comet, one more edge being its head.
	
	We get a 3-graph $H(14)$ on 14 vertices with $|H(14)|\ge 359$, colored by 8 colors, $H(14)=\bigcup_{i=1}^8G_i$. The average number of edges per color is at least 44.9, and hence there exists a color, say $G_8$, with $|G_8|\ge 45$. Similarly as in the previous step we reduce the picture to
a 3-graph $H(13)$ on 13 vertices with $|H(13)|\ge 280$, colored by 7 colors, $H(13)=\bigcup_{i=1}^7G_i$.
	
	This time the average number of edges per color is at least 40, and therefore, by Theorems \ref{ex1} and \ref{ex2}, either each color is a copy of $\co(13)$ or $K_6\cup K_6\cup K_1$, or there exists a color, say $G_7$, contained in the full star $S_{13}$. We will show in the similar way as before, that $H(13)$ can not by decomposed into edge-disjoint copies of $\co(13)$ and $K_6\cup K_6\cup K_1$, and therefore the later case must occur. Indeed, first notice that there is not enough space for two edge-disjoint copies of $K_6\cup K_6\cup K_1$ in $K_{13}$ and therefore also in $H(13)$. Fixed one copy of $K_6\cup K_6\cup K_1$ in $K_{13}$. By pigeon-hole principle, any other copy of $K_6$ must share at least three vertices with one of the fixed copies of $K_6$ and therefore they are not edge-disjoint. Now observe, that since during our procedure we have lost at most 6 edges of $K_{13}$, for each vertex $v\in V(H(13))$ we have $\deg_{H(13)}(v)\ge {12 \choose 2} - 6 = 60$ and also for each vertex of a comet $\co(13)$ which is not its center, we have $\deg_{\co(13)}(v)\le 8$. Since we can decompose $H(13)$ into at most seven copies of $\co(13)$, there must exist a vertex $v\in V(H(13))$ which is not a center of any of these comets and therefore $\deg_{H(13)}(v)\le 10+ 6\cdot 8 = 58<60$, a contradiction. Consequently we have $G_7 \subseteq S_{13}$ and, by removing the center of this star, we obtain a 6-coloring of a 3-graph $H(12)$ on 12 vertices with $|H(12)|\ge 214$.
	
	To proceed, let us assume for a while, that none of the colors $G_i$, $i \in [6]$, is a  star. Then, by Theorems \ref{ex1}--\ref{ex3}, each color with more than 32 edges is a subset of $K_6\cup K_6$. The average number of edges per color is at least 35.6, and hence there exists a color, say $G_6$, with $G_6\subset K_6\cup K_6$. We remove all edges of this copy of $K_6\cup K_6$, getting a bipartite 3-graph $H'(12)$ with a bipartition $V(H'(12))=V\cup U$, $|V|=|U|=6$, and with $|H'(12)|\ge 174$ edges colored by 5 colors, $H'(15)=\bigcup_{i=1}^5G_i$. Note, that every subgraph of $K_6\cup K_6$ contained in $H'(12)$ (and consequently each color class of $H'(12)$) has at most 36 edges. Since $3\cdot 36 + 2\cdot 32 = 172<174$, at least 3 colors must be subsets of $K_6\cup K_6$ and have at least 34 edges. Now observe, that if two color classes, say $G_1$ and $G_2$, have at least 34 edges each, then they are disjoint unions of two copies of $K_6$, one of the vertex set $U'_i \cup W'_i$, the other one on $U_i''\cup W''_i$, with four missing edges $U_i',U_i'',W_i',W_i''$, where $U=U_i'\cup U_i''$, $V=V'_i\cup V_i''$, $i=1,2$, and $\{U_1',U''_1\}=\{U_2',U''_2\}$ (See Fig. \ref{R2}). 
	\begin{figure}[!ht]
		\centering
		\includegraphics [width=7cm]{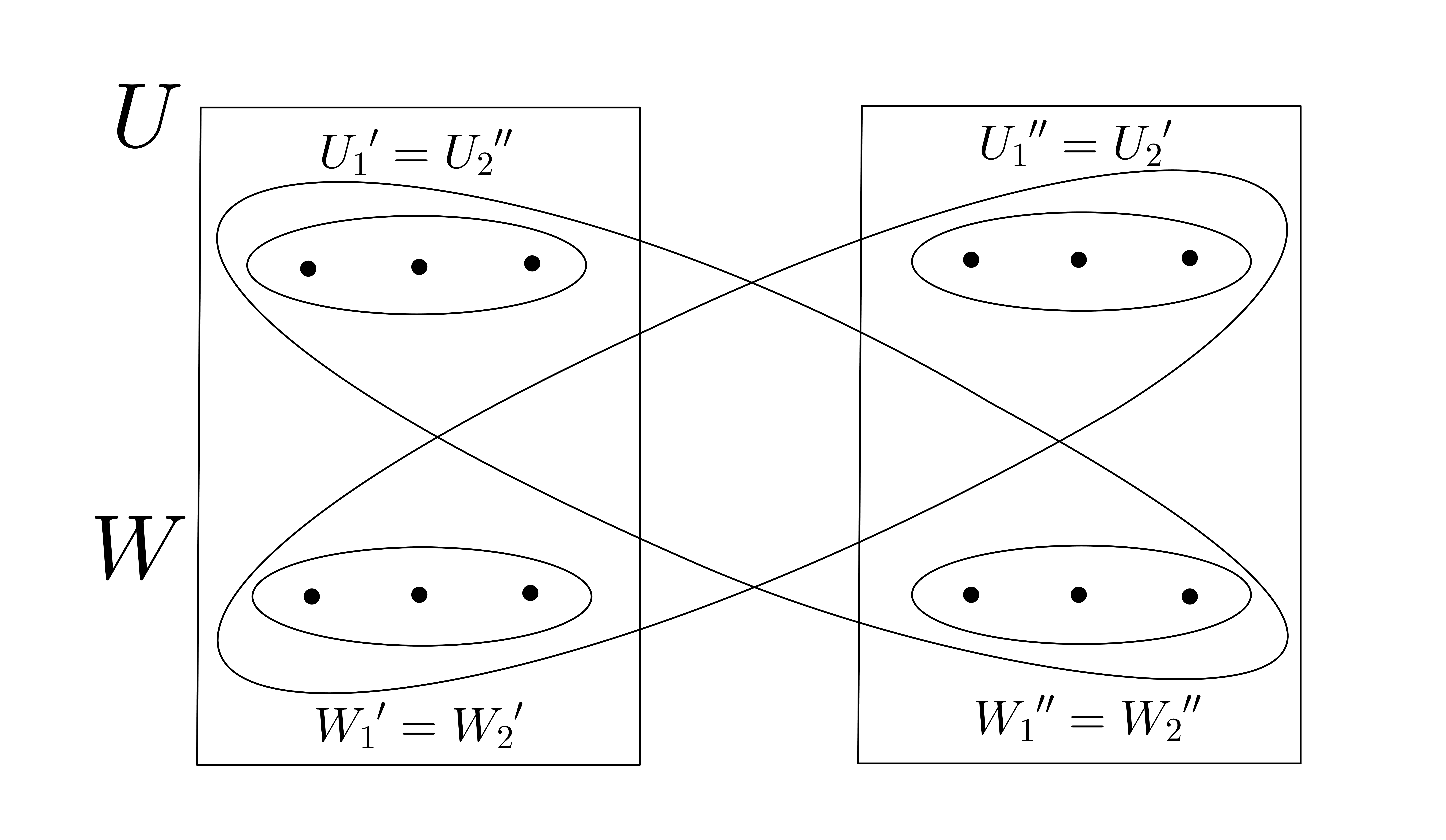}
		\caption{The partition of the set of vertices of $H'(12)$, $G_1$ and $G_2$.} \label{R2}
	\end{figure}
	Otherwise, if $1\le |U_1'\cap U_2'|\le 2$, $G_1$ and $G_2$ would share at least six edges, and thus $|G_1|+|G_2|\le 36+36-6 < 2\cdot 34$. This simply means that one of the partitions, of $U$ or of $W$, must be swapped. But this is impossible for three color classes. Consequently, at least one color, say $G_6$, is a  star. We remove the center of this star to get a 5-coloring of a 3-graph $H(11)$ on 11 vertices with $|H(11)|\ge 159$.

 By repeating this argument three more times, we finally arrive at a 2-coloring of a 3-graph $H(8)=G_1\cup G_2$, with $|H(8)|\ge 50$ which, by Theorem \ref{ex1}, should contain a copy of $P$, a contradiction.
\qed


\section{Proof of Theorem \ref{ex5}}

Let us define $\mathcal{H}_n=\Ex^{(1)}(n;P)\cup \Ex^{(2)}(n;P)\cup \Ex^{(3)}(n;P)\cup \Ex^{(4)}(n;P)$. To prove Theorem~\ref{ex5} we need to find, for each $n\ge 7$, a $P$-free, $n$-vertex 3-graph $H$ with the biggest possible number of edges such that, whenever $G\in \mathcal{H}_n$ then $H\nsubseteq G$. Moreover we will show, that $|H|=h_n$, where $h_n$ is the number of edges, given by the formula to be proved.

 First note that for each $n\ge 7$, all candidates for being 5-extremal 3-graphs do qualify, that is, are $P$-free, are not contained in any of the 3-graphs from $\mathcal{H}_n$, and have $h_n$ edges. To finish the proof, we will show that each $P$-free, $n$-vertex 3-graph $H$, not contained in any of 3-graph from $\mathcal{H}_n$ satisfy $|H| < h_n$ unless it is one of the candidates for being 5-extremal 3-graph itself.

For the latter task, we distinguish  two cases: when $H$ is connected and disconnected. The entire proof is inductive, in the sense that here and there we  apply the very Theorem~\ref{ex5}
for smaller instances of $n$, once they have been confirmed.

Let for all $n\ge 7$, $H$ be $P$-free $n$-vertex 3-graph such that for each $G \in \mathcal{H}_n$, $H\nsubseteq G$. Moreover let $H$ be different from all candidates for being 5-extremal 3-graphs with the same number of vertices. We will show that $|H|<h_n$.

\subsection{Connected case}

	We start with the connected case. First let us assume, that $M\nsubseteq H$ and consider consecutive intersecting families. Recall that for all $n\ge 7$, $H\nsubseteq S_n$, for $7\le n\le 12$, $H\nsubseteq G_1(n)$ and $H\nsubseteq G_2(n)$, for $7\le n\le 9$ and $n=11$, $H\nsubseteq G_3(n)$ and finally, for $n=7$ $H$ is not equal to any of 4-extremal 3-graphs for $M$. Therefore, by Theorems \ref{ntif}, \ref{ntntif} and \ref{ntntntif}, we get that for all $n\ge 7$,
	$$|H| < h_n.$$
	Consequently we will be assuming by the end of the proof, that $M\subset H$. If additionally $C\subset H$, then by Lemma \ref{spojnyzCM}, $H\subseteq K_5^{+(n-5)}$ and hence $|H|\le |K_5^{+(n-5)} |=n+5$. Therefore, for $n\ge 10$, $|H|<h_n$. If $n=7$, as $K_5^{+2}\in \mathcal{H}_7$, we have $H\nsubseteq K_5^{+2}$ and thus we may exclude this case. Lastly, for $8\le n \le 9$, by the definition of $H$, $H\neq K_5^{+(n-5)}$ and hence $|H| < h_n$. Therefore, in the rest of the proof we will be assuming, that $C\nsubseteq H$.
	
	Finally, let $H$ be connected $\{P,C\}$-free 3-graph containing $M$. Then by Lemma \ref{PCM}, for $7\le n \le 8$, $|H|\le 2n-4<h_n$ and for $n=9$, since $H\notin\Ex(9,\{P,C\}|M)$, we have $|H|<14=h_9$.
	
	For $10\le n\le 11$ we need two more facts, which we state here without the proof. Namely $\ex_{conn}(10;\{P,C\}|M)=19$ and $\Ex_{conn}(10;\{P,C\}|M)=\{\co(10)\}$. Since, by the definition of $H$, $H\neq \co(10)$, this implies, that $|H|<19=h_{10}$. Whereas for $n=11$ we have $\ex_{conn}^{(2)}(11;\{P,C\}|M)=18$, and therefore, as $H\nsubseteq \co(11)$, we get $|H|\le \ex_{conn}^{(2)}(11;\{P,C\}|M)=18 < 19=h_{11}$

	Recall, that for all $n\ge 11$, $H\nsubseteq \co(n)$. Moreover,
	for $12 \le n \le 13$, since $|\ro(n)|< h_n$, we may assume, that $H\nsubseteq \ro(n)$. Further, for $n=14$, by the definition of $H$ we have $H\neq \ro(14)$ and thus, if $H\subset \ro(14)$, then $|H|<|\ro(14)|= h_n$. Finally for all $n\ge 15$ we have $H\nsubseteq \ro(n)$.
	Therefore, since for all $n\ge 12$ we have
	$$
	h_n \le {n-6 \choose 2}+10,
	$$
	to complete the proof of the connected case it is enough to prove the following Lemma,
	\begin{lemma}\label{con}
		If $H$ is a connected, $n$-vertex, $n\ge12$, $\{P,C\}$-free 3-graph containing $M$ such that $H \nsubseteq \co(n)$ and $H \nsubseteq \ro(n)$, then $|H|<{n-6 \choose 2}+10$.
	\end{lemma}
	We devote an entire Section \ref{prepa} to prove Lemma \ref{con}.

\subsection{Disconnected case}

 Now let $H$ be disconnected and let $m=m(H)$ be the number of vertices in the smallest componet of $H$. We have $m\neq 2$, since no component of a 3-graph may have two vertices. We now break the proof into several cases.

 Let us express $H$ as a vertex disjoint union of two 3-graphs:

 $$
 H=H'\cup H'', \qquad |V(H')| = m, \quad |V(H'')|=n-m
 $$
 Then, clearly, both $H'$ and $H''$ are $P$-free, and thus
 \begin{equation}\label{ml}
 |H|\le \ex^{(1)}(m;P)+\ex^{(1)}(n-m;P).
 \end{equation}

 Below, to bound $|H|$, we use the Tur\'an numbers for $P$ of the $1^{\textrm{st}}$, $2^{\textrm{nd}}$, $3^{\textrm{rd}}$, $4^{\textrm{th}}$ and $5^{\textrm{th}}$ order and utilize, respectively, Theorems \ref{ex1}, \ref{ex2}, \ref{ex3}, \ref{ex4} and \ref{ex5} (per induction).

Let $v$ be an isolated vertex ($\mathbf{m=1}$). Since for $n=7$ and any 3-graph $H''$, $K_1\cup H''\subseteq K_1\cup K_6\in \mathcal{H}_7$, we may assume that $n\ge 8$.
For $8\le n \le 11$, as $H$ cannot be a sub-3-graph of $S_n$, $K_{6}\cup K_{n-6}$, $G_1(n)$ or $G_2(n)$, $H''$ is not a sub-3-graph of $S_{n-1}$, $K_{6}\cup K_{n-7}$, $G_1(n-1)$ and $G_2(n-1)$. Consequently, for $n=8,10$,
$$
|H| = |H''|\le \ex^{(4)}(n-1;P)<h_n.
$$
 For $n=9$ additionally we have $H''\nsubseteq G_3(8)$ and therefore 
 $$|H|\le \ex^{(5)}(8;P)=13 < 14=h_{9},$$
  whereas for $n=11$, $H''\nsubseteq K_5\cup K_5$ and $H''\nsubseteq \co(10)$. Consequently
$$
|H|=|H''|< \ex^{(5)}(10;P)=19=h_{11}.
$$
 For $n\ge 12$, since $H=K_1\cup H''$ is not a sub-3-graph of any of the 3-graphs in $\mathcal{H}_n$, we have $H''\nsubseteq S_{n-1}$ and $H''\nsubseteq \co(n-1)$. Moreover, for $n=12,13$, $H''\nsubseteq K_6\cup K_{n-7}$, for $n=12$, $H''\nsubseteq G_1(n-1)$ and $H''\nsubseteq G_2(n-1)$, for $n=14$, $H''\nsubseteq 2K_6\cup K_1$, for $n=14,15$, $H''\nsubseteq K_6\cup S_{n-7}$ and finally, for $n\ge 15$, $H''\nsubseteq K_4\cup S_{n-5}$. Consequently,
$$
|H|=|H''|\le \ex^{(4)}(n-1;P)< h_n.
$$

 For $\mathbf{m=3}$
  and $n=7,8$, by (\ref{ml}) we get
  $$|H|\le \ex^{(1)}(3;P)+ \ex^{(1)}(n-3;P)=1+\ex^{(1)}(n-3;P)< h_n,$$
 Since each disconnected 3-graph $H=H'\cup H''$ with $|V(H')|=3$ and $|V(H'')|=6$ is a sub-3-graph of $K_3\cup K_6\in \mathcal{H}_9$, we may assume that $n\neq 9$.
 For $n=10$ we have $K_3\cup K_6\cup K_1 \subset K_4 \cup K_6 \in \mathcal{H}_{10}$. Consequently $H''\nsubseteq K_6\cup K_1$ and thus $|H''|\le \ex^{(2)}(7;P)=15$. Hence $|H|\le 1 + 15 = 16 < 19=h_{10}$.

Further, for all $n\ge 11$, since $\co(n)\in \mathcal{H}_n$, we have $H''\nsubseteq S_{n-3}$. Therefore for $n\ge 12$,
 $$
|H|\le 1+ \ex^{(2)}(n-3;P)<h_n,
$$
whereas, for $n=11$ additionally we have $H\nsubseteq K_3\cup K_6 \cup K_2 \subset K_6\cup K_5\in \mathcal{H}_{11}$. Thus $H''\nsubseteq K_6\cup K_2$ and consequently 
$$|H| \le 1 + \ex^{(3)}(8;P)=17 < 19=h_{11}.$$

For $\mathbf{m=4}$ and $n=8$ by (\ref{ml}) we have 
$$|H|\le \ex^{(1)}(4;P)+\ex^{(1)}(4;P)=4+4=8<h_8.$$ 
For $n=9$, by the definition of $H$, $H\neq K_4\cup K_5$ end therefore $|H|<| K_4\cup K_5|=14=h_9$. Similarly like before, we may skip the case $n=10$, because each disconnected 3-graph $H=H'\cup H''$ with $|V(H')|=4$ and $|V(H'')|=6$ is a sub-3-graph of $K_4\cup K_6 \in \mathcal{H}_{10}$.
	For $n=11$, since $K_4\cup K_6 \cup K_1\subset K_5 \cup K_6 \in \mathcal{H}_{11}$, we have $H''\nsubseteq K_6\cup K_1$ and therefore $|H''|\le \ex^{(2)}(7;P)=15$ with the equality only for $H''=S_{7}$. But, by the definition of $H$, $H\neq K_4\cup S_7$, and hence 
	$$|H|<|K_4\cup S_7|=19=h_{11}.$$
	 Further, for $n=12, 13$, since $\Ex^{(1)}(n-4;P)=\{S_{n-4}\}$ and $H\neq H_4 \cup S_{n-4}$, we have $|H| <|H_4 \cup S_{n-4}|= h_{n}$. Finally, for $n\ge 14$, since $K_4 \cup S_{n-4} \in \mathcal{H}_n$ we get $H''\nsubseteq S_{n-4}$ and consequently,
	 $$
	 |H|\le\ex^{(1)}(4;P)+ \ex^{(2)}(n-4;P)<h_n.$$

Now let $\mathbf{m=5}.$ Notice that each disconnected 3-graph $H=H'\cup H''$ with $|V(H')|=5$ and $5\le |V(H'')|\le 6$ is a sub-3-graph of $K_5\cup K_5\in \mathcal{H}_{10}$ and $K_5 \cup K_6\in \mathcal{H}_{11}$ respectively. Therefore we may consider only $n\ge 12$. For $n=12$, since $K_5\cup K_6\cup K_1 \subset K_6 \cup K_6 \in \mathcal{H}_{12}$, we have $|H''|\le \ex^{(2)}(7;P)=15$ with the equality only for $H''=S_7$. But, by the definition of $H$, $H\neq K_5\cup S_7$ and hence $|H|<|K_5\cup S_7|= 25=h_{12}$.
Finally, for $n\ge 13$, by (\ref{ml}),
$$
|H|\le  \ex^{(1)}(5;P)+ \ex^{(1)}(n-5;P)= 10+ \binom{n-6}{2}\le h_n,
$$
where the equality is achieved only by the candidates for 5-extremal 3-graphs with the proper number of vertices.

For $\mathbf{m=6}$ we have $n\ge 12$, but as each disconnected 3-graph $H'\cup H''$ with $|V(H')|=|V(H'')|=6$ is a sub-3-graph of $K_6\cup K_6\in \mathcal{H}_{12}$, we may consider only $n\ge 13$. Recall, that $\{2K_6\cup K_1, K_6\cup S_7, K_6\cup G_1(7), K_6\cup G_2(7)\}\subset \mathcal{H}_{13}$ and therefore, for $n=13$, $H''$ is not contained in any of the 3-graphs $K_6\cup K_1, S_7, G_1(7), G_2(7)$. Consequently, $|H''| \le \ex^{(4)}(7;P) = 12$ with the equality only for $H''=G_3(7)$ and $H''=K_5^{+2}$. But, by the definition of $H$, $H\neq K_6\cup K_5^{+2}$ and $H\neq K_6\cup G_3(7)$ and thus 
$$|H|< |K_6\cup K_5^{+2}| =|K_6\cup G_3(7)|= h_{13}.$$
For the same reason, if $n=14$, then  $H'' \nsubseteq S_8$ and $H''\nsubseteq K_6\cup K_2$. Consequently,  
$$|H|=|H'|+|H''|\le \ex^{(1)}(6;P)+\ex^{(3)}(8;P)=20+16<39=h_{14},$$
 whereas for $n=15$, we have $H''\nsubseteq S_{9}$ and hence 
 $$|H|\le\ex^{(1)}(6;P)+ \ex^{(2)}(9;P)=20+21<46=h_{15}.$$
Further, for $n=16,17$, by the definition of $H$, $H\neq K_6\cup S_{n-6}$. Consequently, as $\Ex(n-6;P)=\{S_{n-6}\}$, we get
 $$
 |H|<|K_6\cup S_{n-6}|=h_n.
 $$
Finally, for $n\ge 18$, by (\ref{ml}),
$$
\begin{aligned}
|H|\le \ex^{(1)}(6;P)+ \ex^{(1)}(n-6;P)=20+ \binom{n-7}{2} <\binom{n-6}{2} +10=h_n.
\end{aligned}
$$

If $\mathbf{m=7}$, then $n\ge 14$.  For $n=14$, since $H\nsubseteq 2K_6\cup 2K_1\in \mathcal{H}_{14}$, at least one of the components of $H$ is not a sub-3-graph of $K_6\cup K_1$ and therefore has at most $\ex^{(2)}(7;P)=15$ edges. Consequently,
$$
|H|\le  \ex^{(1)}(7;P)+ \ex^{(2)}(7;P)= 20+ 15<39= h_{14}.
$$
To bound the number of edges of $H$ for $n\ge 15$ we use (\ref{ml}) to get
$$
|H|\le  \ex^{(1)}(7;P)+ \ex^{(1)}(n-7;P)= 20+ \binom{n-8}{2}<\binom{n-6}{2}+10\le h_n.
$$

Finally, for $\mathbf{m\ge 8}$ we have $n\ge 16$ and, by (\ref{ml}),
$$
\begin{aligned}
|H|\le \ex^{(1)}(m;P)+ \ex^{(1)}(n-m;P)=\binom{m-1}{2}+ \binom{n-m-1}{2}\\\le\binom 72+\binom{n-9}2 <\binom{n-6}{2} +10\le h_n.
\end{aligned}
$$

\bigskip

\section{The proof of Lemma \ref{con}}\label{prepa}
Recall that $H$ is a connected, $n$-vertex, $n\ge12$, $\{P,C\}$-free 3-graph such that  $M\subset H$, $H \nsubseteq \co(n)$ and $H \nsubseteq \ro(n)$. We need to show that
$$
|H|<{n-6 \choose 2}+10.
$$

Since for $n\ge 11$, by Lemma \ref{pcppm}
$$
\ex(\{n;P,C,P_2\cup K_3\}|M)=2n-4 < {n-6 \choose 2}+10,
$$
 we may assume that $ P_2\cup K_3\subset H$.
Let us denote a copy of $P_2$ from $P_2\cup K_3$ in $H$ by $Q$ and the vertex of degree two in $Q$ by $x$. We let $U=V(Q)$, $V=V(H)$ and $W=V\setminus U$.
  Moreover, let $W_0$ be the set of vertices of degree zero in $H[W]$ and $W_1=W\setminus W_0$.
(see Fig. \ref{FigR6}).  Note that, by definition,  $H[W]=H[W_1]$ and $|W_1|\ge3$.

\bigskip
\begin{figure}[!ht]
\centering
\includegraphics [width=9cm]{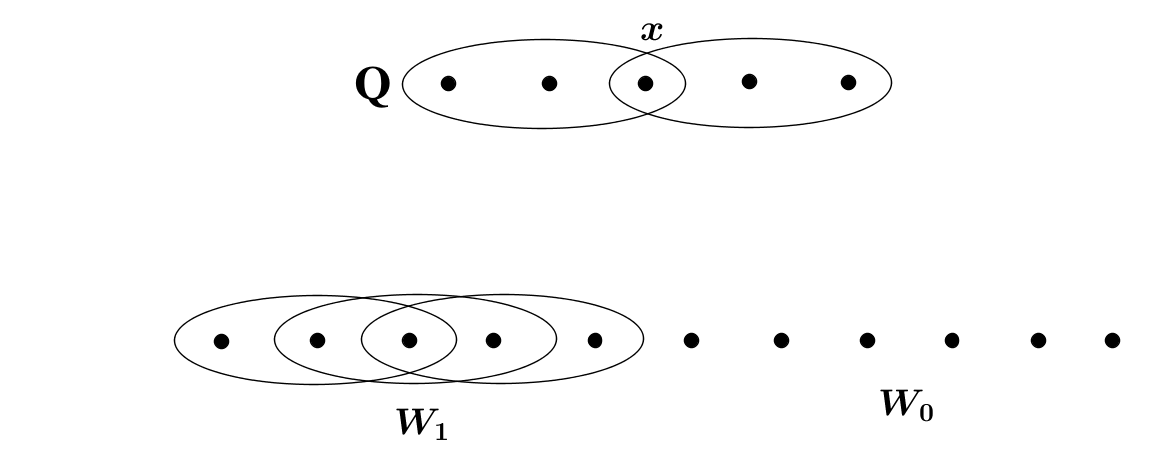}
\caption{Set-up for the proof of Lemm \ref{con}} \label{FigR6}
\end{figure}

\medskip

\noindent We also split the set of edges of $H$.
First, notice that, since $H$ is $P$-free, there is no edge with one vertex in each $U$, $W_0$, and $W_1$. We define $H_i=\{h\in H: h\cap U\neq \emptyset, h\cap W_i\neq \emptyset\}$, where $i=0,1$. Then, clearly,
\begin{equation}\label{HHH}
    H=H[U]\cup H[W]\cup H_0\cup H_1,
\end{equation}
with all four parts edge-disjoint.
Since by definition $H[U]\cup H_0=H[U\cup W_0]$, sometimes we will use the following equality
\begin{equation}\label{HH}
H=H[U\cup W_0]\cup  H_1\cup H[W].
\end{equation}
Recall that $H$ is $C$-free, and therefore one can use Theorem \ref{c3} to get the bounds, for $|W_0|\ge 1$
\begin{equation}\label{uwstar}
|H[U\cup W_0]| \le \binom{|U\cup W_0|-1}{2}=\binom{|W_0|+4}{2}
\end{equation}
and for $|W_1|\ge 6$,
\begin{equation}\label{hwstar}
|H[W]|\le \binom{|W_1|-1}{2}
\end{equation}

Notice that for each edge $h\in H_0\cup H_1$ with $|h\cap U|=1$ we have $h\cap U = \{x\}$, because otherwise $h$ together with $Q$ would form a copy of $P$ in $H$. We let
$$
F^0=\{h\in H_0\cup H_1: h\cap U=\{x\}\}.
$$
Also, to avoid a copy of $C$ in $H$, if for $h\in H_0\cup H_1$ we have $|h\cap U|=2$ then the pair $h\cap U$ is contained in an edge of $Q$. For $k=1,2$, we define
$$
F^k=\{h\in H_0\cup H_1: \quad |h\cap U\setminus\{x\}|=k\}.
$$
Clearly, $H_0\cup H_1=F^0\cup F^1\cup F^2$ (see Fig.
\ref{FigR7}). Further, for $i=0,1$ and $k=0,1,2$, we set
$$F_i^k=F^k\cap H_i.$$
\bigskip
\begin{figure}[!ht]
\centering
\includegraphics [width=7cm]{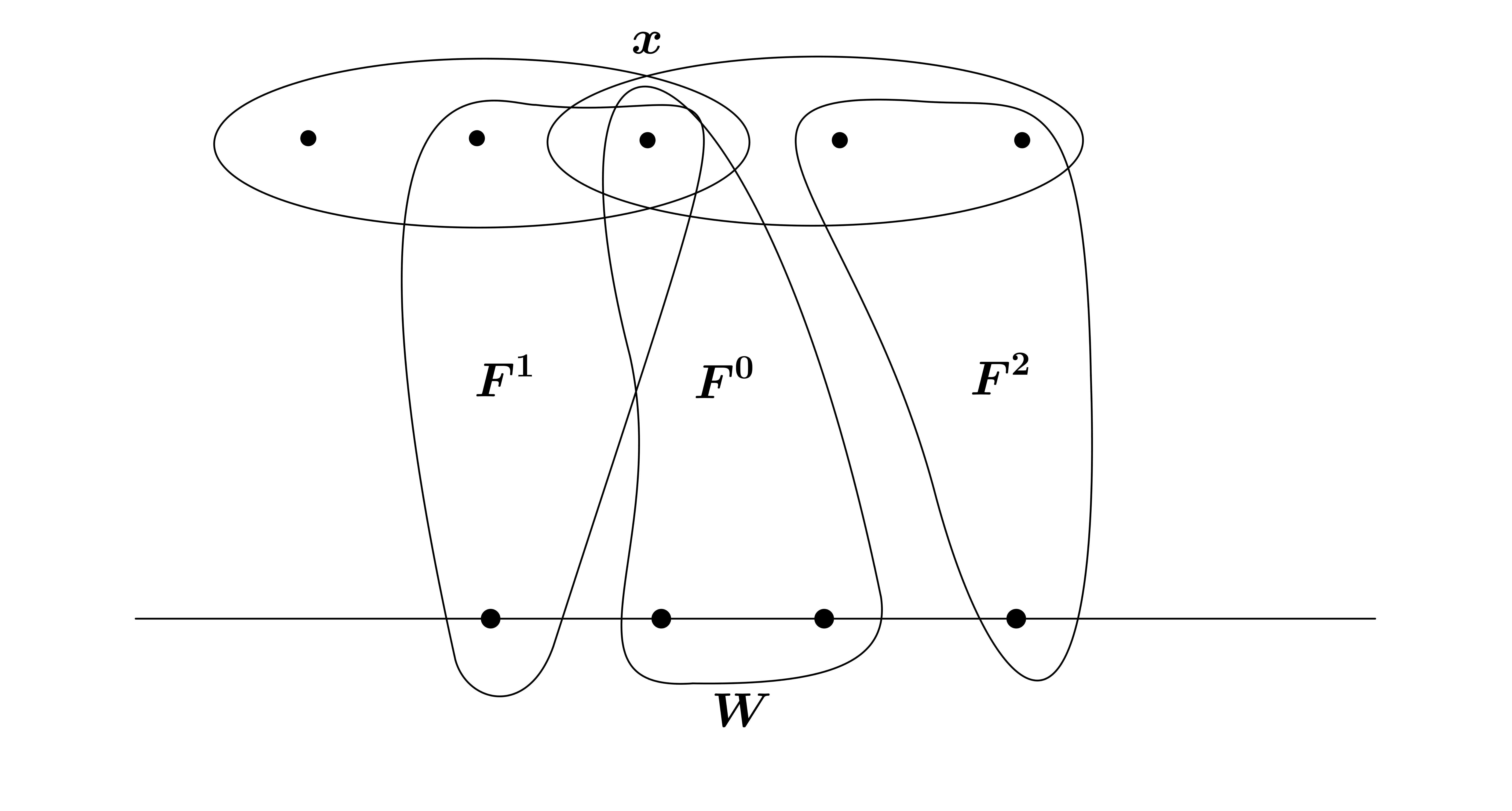}
\caption{Three types of edges in $H_0\cup H_1$} \label{FigR7}
\end{figure}
It was noticed in \cite{JPRr} that, as $H$ is $P$-free, $F^1_1=\emptyset$ and therefore,
\begin{equation}\label{r6}
H_1=F_1^0\cup F_1^2.
\end{equation}
Moreover, for all $v\in W$ we have
\begin{equation}\label{r2}
F^0(v)=\emptyset \quad\textrm{or}\quad F^2(v)=\emptyset,
\end{equation}
and, by the definition of $F^1$ and $F^2$,
\begin{equation}\label{r3}
|F^1(v)|\le 4 \quad \textrm{and}\quad |F^2(v)|\le 2.
\end{equation}
where for a given subset of edges $G\subseteq H$ and for a vertex $v\in V(H)$ we set $G(v)  = \{h\in G: v\in h\}$.

In the whole proof we will be using the fact, that for all edges $e\in F^0$, the pair $e\cap W_1$ is \emph{nonseparable} in $H[W]$, that is, every edge of $H[W]$ must contain both these vertices or none. Consequently, for each $v\in W_0$, $|F^0(v)|\le |W_0|-1$ and thus, by (\ref{r2}) and (\ref{r3}),
\begin{equation}\label{hv}
|H(v)| =|F^0(v)|+|F^1(v)|+|F^2(v)|\le 4 + \max\{2,|W_0|-1\}.
\end{equation}

Observe also that, because $H$ is connected, $H_1\neq\emptyset$.
 Consequently, since the presence of any edge of $H_1$ forbids at least 4 edges of $H[U]$,
 \begin{equation}\label{hu1}
 |H[U]|\le 6.
 \end{equation}
\medskip

Moreover, in \cite{JPRr} the authors have proved the following bounds on the number of edges in~$H_1$:

\begin{equation}\label{r5}
\textrm{For}\quad  |W_1|\ge 4, \quad |F^2_1|\le 2|W_1|-4.
\end{equation}

\begin{equation}\label{r4}
\textrm{For}\quad  |W_1|\ge 3, \quad |H_1|\le 2|W_1|-3.
\end{equation}

\begin{equation}\label{r1}
\textrm{For}\quad  |W_1|\ge 3, \quad |F_1^0|\le |W_1|.
\end{equation}
As a consequence of these inequalities one can prove the following
\begin{equation}\label{r7}
	\textrm{For}\quad  |W_1|\ge 7, \quad |H[U]|+ |H_1|\le 2|W_1|-1.
\end{equation}
Indeed, if $|H_1|\le |W_1|$, then (\ref{r7}) results from (\ref{hu1}) and the inequality $|W_1|-1\ge 7-1=6$. Otherwise, by (\ref{r1}), (\ref{r6}) and (\ref{r2}), there exists a vertex $v\in W_1$, such that $|F_1^2(v)|=2$. As $H$ is $\{P,C\}$-free, by the definition of $F_1^2(v)$, this implies, that $|H[U]|=2$ and (\ref{r7}) follows from (\ref{r4}).

We also need the following fact proven in \cite{PR}.

\begin{fact}\cite{PR}\label{UW0} If  $F^2_1\neq \emptyset$, then
	\begin{equation}\label{e4}
	|H[U\cup W_0]|\le
	\left\{\begin{array}{ll}
	8 &\textrm{for } |W_0|=1,\\
	3|W_0|+7 &\textrm{for } 2\le |W_0|\le 4,\\
	\binom{|W_0|+2}{2}+1 &\textrm{for } |W_0|\ge 5.
	\end{array}\right.
	\end{equation}
\end{fact}

\bigskip

	We split the whole proof of Lemma \ref{con} into a few short parts, Facts \ref{w0empty}-\ref{w05}.

\begin{fact}\label{w0empty}
	For $n\ge 13$ if $W_0=\emptyset$ and $H_1\neq\emptyset$, then  $|H|< 10 +{n-6\choose 2}$.
\end{fact}

\begin{proof}
	Let us consider two cases, whether or not $H[W]\subseteq S_{n-5}$. If $H[W]\subseteq S_{n-5}$ then, since $H$ is $P$-free, by (\ref{r3}), $|F^2|=|F^2(y)|\le 2$ where $y\in W_1$ is the center of the star $S_{n-5}$. Additionally if $F_1^0=\emptyset$, then by (\ref{HHH}),  (\ref{hu1}), (\ref{r6}) and (\ref{hwstar}),
	$$
	|H|= |H[U]| + |H_1| + |H[W]| \le 6 + 2 + {n-6\choose 2} =  {n-6\choose 2} + 8< {n-6\choose 2} + 10.
	$$
	Otherwise $F_1^0\neq\emptyset$. As for each $h\in F_1^0$, the pair $h\cap W_0$ is nonseparable, one can show, that $|H[W]|\le {n-8 \choose 2}+1$. By (\ref{r1}), $|F_1^0|\le |W_1|=n-5$ and hence by (\ref{r6}), $|H_1|\le n-5 + 2 = n-3$. Consequently, by (\ref{HHH}) and (\ref{hu1}),
	$$
	|H|= |H[U]| + |H_1| + |H[W]| \le 6 + n-3 + {n-8\choose 2} +1=  {n-7\choose 2} +12< {n-6\choose 2} + 10.
	$$
	
	Now we move to the case $H[W]\nsubseteq S_{n-5}$ and use Theorem \ref{ex1} to bound the number of edges in $H[W]$ by $\ex^{(2)}(n-5;P)$. Moreover, by (\ref{r7}), $|H[U]| + |H_1| \le2 (n-5) - 1 = 2n-11$. Consequently, by (\ref{HHH}) and Theorem \ref{ex2},
	$$
	|H|= |H[U]| + |H_1| + |H[W]| \le 2\cdot n -11 +\ex^{(2)}(n-5;P)<{n-6\choose 2} + 10,
	$$
	where the last inequality is valid for $n\ge 14$.
	For $n=13$ we have to strengthen the bound of $H[W]$. As $H[W]\nsubseteq K_6\cup K_2$, $H[W]\neq G_1(8)$ and $H[W]\neq G_2(8)$, by Theorems \ref{ex1}, \ref{ex2} and \ref{ex3} we have $|H[W]|< \ex^{(3)}(8;P)=16$ and therefore
	$$
	|H| < 2\cdot 13 -11 +16=31= {13-6\choose 2} + 10.
	$$

\end{proof}

\bigskip

\begin{fact}\label{W13}
	For $n\ge 13$  if $H_1\neq\emptyset$, $H\nsubseteq \co(n)$ and $|W_1|=3$ then  $|H|< 10 +{n-6\choose 2}$
\end{fact}

\begin{proof}
	We have $|H[W]|=1$, $|U\cup W_0|=n-3$ and by (\ref{r4}), $|H_1|\le 3$. Therefore, by (\ref{HH}),
	$$
	|H|= |H[U\cup W_0]| + |H_1| + |H[W]| \le |H[U\cup W_0]|+ 3+ 1 = |H[U\cup W_0]|+4.
	$$
	Consequently all we need to do is to bound the number of edges in $H[U\cup W_0]$.
	Since $H\nsubseteq \co(n)$, either $F_1^2\neq\emptyset$ or $H[U\cup W_0]\nsubseteq S_{n-3}$. In the former case we use Fact \ref{UW0} to get $ |H[U\cup W_0]| \le {n-6\choose 2} + 1$ and therefore
	$$
	|H|\le {n-6\choose 2} + 1 + 4 =  {n-6\choose 2} + 5< {n-6\choose 2} + 10.
	$$
	Otherwise $H[U\cup W_0]\nsubseteq S_{n-3}$, so by Theorem \ref{ex1}, $|H[U\cup W_0]| \le ex^{(2)}(n-3;P)$. Consequently, by Theorem \ref{ex2}, for $13\le n\le 15$, $|H[U\cup W_0]| \le 20 + {n-3 -6\choose 3}$ and therefore,
	$$
	|H|\le 20 + {n-9\choose 3} + 4 = {n-9\choose 3} + 24 <  {n-6\choose 2} + 10,
	$$
	Whereas for $n\ge 16$ we get $|H[U\cup W_0]| \le 4+ {n-3-4\choose 2}$, and hence
	$$
	|H|\le {n-7\choose 2} + 4 + 4 =  {n-7\choose 2} + 8< {n-6\choose 2} + 10.
	$$
\end{proof}

\bigskip

\begin{fact}\label{W14}
	For $n\ge 13$, if $H_1\neq\emptyset$, $H\nsubseteq \ro(n)$ and $|W_1|=4$ then  $|H|< 10 +{n-6\choose 2}$
\end{fact}

\begin{proof}
	The proof goes along the lines of the previous one. We have $|H[W]|\le {4\choose 3}=4$, $|U\cup W_0|=n-4$ and by (\ref{r4}), $|H_1|\le 5$.  Therefore, by (\ref{HH}),
	$$
	|H|= |H[U\cup W_0]| + |H_1| + |H[W]| \le |H[U\cup W_0]|+ 5+ 4 = |H[U\cup W_0]|+9.
	$$
	Consequently to finish the proof we need to bound $|H[U\cup W_0]|$. Since $H\nsubseteq \ro(n)$, either $F_1^2\neq\emptyset$ or $H[U\cup W_0]\nsubseteq S_{n-4}$. In the former case we use Fact \ref{UW0} to get for $n=13$, $|H[U\cup W_0]|\le 19$ and consequently
	$$
	|H|\le 19+9=28<31=10 + {13-6\choose 2}.
	$$
	Whereas for $n\ge 14$, $|H[U\cup W_0]|\le {n-7\choose 2} + 1$ and hence,
	$$
	|H| \le {n-7\choose 2} + 1 + 9 =  {n-7\choose 2} + 10< {n-6\choose 2} + 10.
	$$
	Otherwise $H[U\cup W_0]\nsubseteq S_{n-3}$ so we use Theorem \ref{ex1} to get $|H[U\cup W_0]|\le \ex^{(2)}(n-4;P)$. Consequently, by Theorem \ref{ex2}, for $13\le n\le 16$, $H[U\cup W_0]|\le 20 + {n-4-6\choose 3}$ and hence
	$$
	|H|\le 20 + {n-10\choose 3} + 9 = {n-10\choose 3} + 29<  {n-6\choose 2} + 10.
	$$
	Whereas for $n\ge 17$ we have $ |H[U\cup W_0]|\le 4+ {n-4-4\choose 2}$ and therefore,
	$$
	|H| \le 4+{n-8\choose 2}  + 9 =  {n-8\choose 2} + 13< {n-6\choose 2} + 10.
	$$

\end{proof}

\bigskip

\begin{fact}\label{n12}
	If $n=12$, $H_1\neq \emptyset$ and $H\neq \co(12)$ then  $|H|< 10 +{12-6\choose 2}=25$.
\end{fact}

\begin{proof} Let us split the proof into five parts according to the size of the set $W_1$. We start with
	$\mathbf{|W_1|=3}$. Then $|W_0|=4$, $|U\cup W_0|=9$, $|H[W]|=1$ and by (\ref{r4}), $|H_1|\le 3$. Consequently, by (\ref{HH}),
	$$
	|H|= |H[U\cup W_0]| + |H_1| + |H[W]| \le |H[U\cup W_0]|  + 3 + 1 = |H[U\cup W_0]| +4.
	$$
Further, as  $H\nsubseteq \co(12)$, either $F_1^2\neq\emptyset$ or $H[U\cup W_0]\nsubseteq S_{n-3}$. In the former case we use Fact \ref{UW0} to get $|H[U\cup W_0]|\le 19$. Otherwise, $H[U\cup W_0]\nsubseteq S_{n-3}$, and since $H[U\cup W_0]\neq K_6\cup K_3$, by Theorems \ref{ex1} and \ref{ex2}, $|H[U\cup W_0]| < 21$. In both cases $|H[U\cup W_0]| \le 20$ and therefore
	$$
		|H|\le |H[U\cup W_0]| + 4 \le 20 +4=24 < 25.
	$$
	
	For $\mathbf{|W_1|=4}$ we have $|W_0|=3$, $|U\cup W_0|=8$ and $|H[W]|\le {4\choose 3}=4$. If $F_1^2= \emptyset$, then $H_1=F_1^0\neq\emptyset$ and as for each $h\in F_1^0$ the pair $h\cap W_1$ is nonseparable, $|H_1|=1$ and $|H[W]|=2$. Consequently, by (\ref{HH}) and (\ref{uwstar}),
	$$
	|H|= |H[U\cup W_0]| + |H_1| + |H[W]| \le {7 \choose 2} + 1 + 2 = 24 < 25.
	$$
			Otherwise $F_1^2\neq\emptyset$ and we can use Fact \ref{UW0} to get $|H[U\cup W_0]|\le 16$. For $F_1^0\neq\emptyset$, $|H[W]|=2$ and consequently, by (\ref{HH}) and (\ref{r4}),
			$$
			|H|= |H[U\cup W_0] + |H_1| + |H[W]| \le 16 + 5 + 2 = 23 < 25.
			$$
			Whereas for $F_1^0=\emptyset$ we use (\ref{HH}), (\ref{r6}) and  (\ref{r5}) to get
			$$
			|H|= |H[U\cup W_0] + |H_1| + |H[W]| \le 16 + 4 + 4 = 24 < 25.
			$$

	Now let $\mathbf{|W_1|=5}$, $|W_0|=2$, $|U\cup W_0|=7$ and $|H[W]|\le {5\choose 3}=10$. For $F_1^2\neq \emptyset$, by Fact \ref{UW0} we get $|H[U\cup W_0]|\le 13$ and moreover $|H[W]|\le 6$, because otherwise we wouldn't be able to avoid a path $P$ in $H$. If additionally  $P_2\subseteq H[W]$ then again by $P\nsubseteq H$, $|H_1|=|F_1^0|+|F_1^2|\le 2+2= 4$. Hence, by (\ref{HH})
	$$
	|H|= |H[U\cup W_0]| + |H_1| + |H[W]| \le 13 + 4 +6=23<25.
	$$
	Otherwise $P_2\nsubseteq H[W]$ and consequently one can show, that $|H[W]|\le 4$. Therefore, by (\ref{HH}) and (\ref{r4}),
	$$
	|H|= |H[U\cup W_0]| + |H_1| + |H[W]| \le 13 + 7 +4=24<25.
	$$
	For $F_1^2= \emptyset$ we have $F_1^0\neq \emptyset$. Hence, since for each $h\in F_1^0$ the pair $h\cap W_1$ is nonseparable, $|H[W]|\le 4$ and $|H_1|=|F_1^0|\le 2$. Consequently, by (\ref{HH}) and (\ref{uwstar}),
		$$
		|H|= |H[U\cup W_0]| + |H_1| + |H[W]| \le  {7-1\choose 2} + 2 +4=21<25.
		$$
		
	We move to $\mathbf{|W_1|=6}$. Then $|W_0|=1$, $|U\cup W_0|=6$ and by (\ref{hwstar}), $|H[W]|\le {6-1\choose 2}=10$.  Let us again start with the case $F_1^2\neq \emptyset$. By (\ref{e4}) we get $|H[U\cup W_0]|\le 8$. If $P_2\subseteq H[W]$ then since $H$ is $P$-free, $|H_1|=|F_1^0|+|F_1^2|\le 2+4= 6$. Consequently, by (\ref{HH}),
			$$
			|H|= |H[U\cup W_0]| + |H_1| + |H[W]| \le 8 + 6 +10=24<25.
			$$
	Otherwise $P_2\nsubseteq H[W]$ and therefore one can show that $|H[W]|\le 6$. By (\ref{r4}), $|H_1|\le 9$ and consequently by (\ref{HH}),
				$$
				|H|= |H[U\cup W_0]| + |H_1| + |H[W]| \le 8 + 9  + 6=23<25.
				$$
				For $F_1^2= \emptyset$ we have $F_1^0\neq\emptyset$. Hence since for each $h\in F_1^0$ the pair $h\cap W_1$ is nonseparable, $|H[W]|\le 8$ and by (\ref{r1}), $|H_1|=|F_1^0|\le 6$.
				Therefore, by (\ref{HH}) and (\ref{uwstar}),
				$$
				|H|= |H[U\cup W_0]| + |H_1| + |H[W]| \le 10 +6  + 8=24<25.
				$$
						
	Finally, $\mathbf{|W_1|=7}$, $W_0=\emptyset$ and by (\ref{hwstar}), $|H[W]|\le {7-1\choose 2}=15$. If $H[W]\subseteq S_7$ then as $H$ is $P$-free, by (\ref{r3}), $|F_1^2|=|F_1^2(y)|\le 2$, where $y\in W_1$ is the center of the star $S_7$. If additionally $F_1^0=\emptyset$, then by (\ref{HHH}), (\ref{r6}) and (\ref{hu1}),
	$$
	|H|= |H[U]| + |H_1| + |H[W]| \le 6 +2  + 15=23<25.
	$$
	Otherwise $F_1^0\neq\emptyset$ and hence $|H_1|=|F_1^0|+|F_1^2|\le 3+2=5$ and $|H[W]|\le 7$. Again we use the fact that for each $h\in F_1^0$ the pair $h\cap W_1$ is nonseparable. Therefore by (\ref{HHH}) and (\ref{hu1}),
	$$
	|H|= |H[U]| + |H_1| + |H[W]| \le 6 +5  + 7=18<25.
	$$
	The last case we have to consider is $H[W]\nsubseteq S_7$. If $M\subseteq H[W]$, then by Lemma \ref{PCM}, $|H[W]|\le \ex(7;\{P,C\}|M)=10$. Otherwise by Lemma \ref{2MC}, $|H[W]|\le \ex^{(2)}(7;\{M,C\})=10$. Hence by (\ref{HHH}) and (\ref{r7}),
	$$
	|H|= |H[U]| + |H_1| + |H[W]| \le 13  + 10=23<25.
	$$
\end{proof}

\bigskip

\begin{fact}\label{w05}
	For $n\ge 12$ if $|W_1|\ge 5$ and $H_1\neq\emptyset$ then
	\begin{equation}\label{z5}
	|H|< {n-6\choose 2}+10.
	\end{equation}
\end{fact}

\begin{proof}
	The proof is by induction on $n$ with the initial step $n=12$ done in Fact \ref{n12}. Let $n\ge 13$. For $W_0=\emptyset$ the inequality (\ref{z5}) results from Fact \ref{w0empty}. Otherwise there exist a vertex $v\in W_0$. Notice, that since $|W_1|\ge 5$ we have $|W_0|\le n-10$ and
consequently, by (\ref{hv}), $|H(v)|\le 4+\max\{2,|W_0|-1\}\le 4 + n-11=n-7$. Finally, by the induction assumption we get $|H-v|< {n-7\choose 2}+10$.
	Therefore,
	$$
	|H|=|H(v)|+|H-v|< n-7+{n-7\choose 2}+10={n-6\choose 2}+10.
	$$
\end{proof}

\end{document}